\documentclass[preprint,11pt]{elsarticle}

\usepackage{amssymb,latexsym,amsmath}     
\usepackage{epsf}
\usepackage{array}
\usepackage{amssymb,amscd}

\usepackage{relsize}%
\usepackage{calc}
\usepackage[mathscr]{eucal}
\usepackage{hhline}
\usepackage{tikz}
\usepackage{multirow}

\usepackage{amssymb,amsmath,amstext,amsgen,amsfonts,amsthm,mathrsfs,verbatim,url,hyperref,mathdots}
\usepackage{enumitem}
\usepackage{pgfplots}
\pgfplotsset{compat=1.15}
\usepackage{mathrsfs}
\usetikzlibrary{arrows}

\usepackage{arydshln}
\newlist{inparaenum}{enumerate}{2}
\setlist[inparaenum]{nosep}
\setlist[inparaenum,1]{label=\bfseries\alph*.}

\newtheorem{theorem}{Theorem}[section]

\newtheorem{example}[theorem]{Example}

\newtheorem{lemma}[theorem]{Lemma}

\newtheorem{remark}[theorem]{Remark}
\newtheorem{proposition}[theorem]{Proposition}

\def\comment#1{}

\def\invddots{\mathinner{\mskip1mu\raise1pt\vbox{\kern7pt\hbox{.}}\mskip2mu
		\raise4pt\hbox{.}\mskip2mu\raise7pt\hbox{.}\mskip1mu}}

\journal{TBD}

\begin{document}

\begin{frontmatter}



\title{Gohberg-Kaashoek Numbers and Forward Stability of the Schur Canonical Form}

 \author[label1]{A. Minenkova, E. Nitch-Griffin, V. Olshevsky}
 \address[label1]{University of Connecticut}



\begin{abstract}
In  the present paper we complete the stability of Schur decomposition depending on the Jordan structure of the perturbation, started in \cite{MNO}.
\end{abstract}


\end{frontmatter}


\section{Introduction}

 	For every quadratic matrix $A$ there is a decomposition $A=UTU^*$ where $U$ is unitary and $T$ is upper triangular. This triangular matrix $T$ is called a \textit{Schur Triangular form} and the factorization is called the \textit{Schur Decomposition}.
 	
In \cite{MNO} it was shown that in general there is no forward stability for the Schur form. 
\begin{proposition}[Different Gohberg-Kaashoek Numbers]\label{GKnonstable} Let us fix matrix $A_0$ and its Schur decomposition  $A_0=U_0T_0U_0^*$.
There exists $K>0$ such that in any neighborhood of $A_0$, i.e. $\{A:\|A-A_0\|<\varepsilon\}$ for any $\varepsilon>0$,
\begin{equation}\label{nonstable}
 \underset{A}{\sup}\ \underset{U,T}{\inf}\ {\|U-U_0\|+\|T-T_0\|}>M>0,
\end{equation} where the supremum is taken over all $A$ in this neighborhood having different Gohberg-Kaashoek numbers from $A_0$ and the infimum is taken over all their Schur factorizations $A=UTU^*$.
\end{proposition}

Thus the natural question to ask is when the forward suitability is possible. We give the answer  in this paper.

We start with the structure of the original matrix. It is of interest to find the classes of matrices $A_0$ for which the Schur form is always stable. It turns out, one such case is when  $A_0$ is non-derogatory, i.e. only one Jordan block per eigenvalue. We discovered that the following result holds.

\begin{theorem}[Non-derogatory case. H\"older forward stability]\label{nonderogatory}
	Let $A_0$ be a non-derogatory matrix. Then its Schur canonical form is forward H\"older stable, i.e. if $A_0=U_0 T_0U_0^*$ then there
	exist  constants $K,\varepsilon>0$ $($depending on $A_0$ only$)$ such that for all $A$ with $\|A-A_0\|<\varepsilon$ there exists a Schur factorization $U TU^*$ of $A$ such that we have
	\begin{equation*}
\|U-U_0\|+\|T-T_0\|\leq K\|A-A_0\|^{1/n}.
\end{equation*}
\end{theorem}

What happens if we consider the structure preserving perturbation? Before answering this question let us talk about how to quantify the structure of a matrix and introduce a concept that will serve this purpose,
the Gohberg-Kaashoek numbers
(see \cite{GK78,O89,MO} for details).

	Let us start by considering $A\in\mathbb{C}^{n\times n}$. Denote by $\sigma(A)$ the set of all its eigenvalues and  by $m_1(A,\lambda)\le m_2(A,\lambda)\le\dots\le m_t(A,\lambda)$ the sizes of all blocks corresponding to $\lambda\in\sigma(A)$ in the Jordan form of $A$. We set $m_i(A,\lambda)=0$ ($i=t+1,\dots,n$) for convenience. The numbers
	$$m_i(A)=\sum_{\lambda\in\sigma(A)}m_i(A,\lambda)$$
	are called \textit{the Gohberg-Kaashoek (GK) numbers}. 

Let $m=\begin{bmatrix}
m_1,m_2, \ldots,m_n
\end{bmatrix}^\top$ be a vector with integer entries with  $m_i\geq m_{i+1}$ for $i=1,\dots,n-1$. The vector $k=\begin{bmatrix}
k_1,k_2,\ldots,k_n
\end{bmatrix}^\top $ with
$
k_i=\underset{1\leq l\leq n}{\max}\{l:m_l\geq i\}
$  is called {\textit{dual}} to $m$.

In terms of the Gohberg-Kaashoek numbers $m_j$'s it means that if we consider the following matrix 
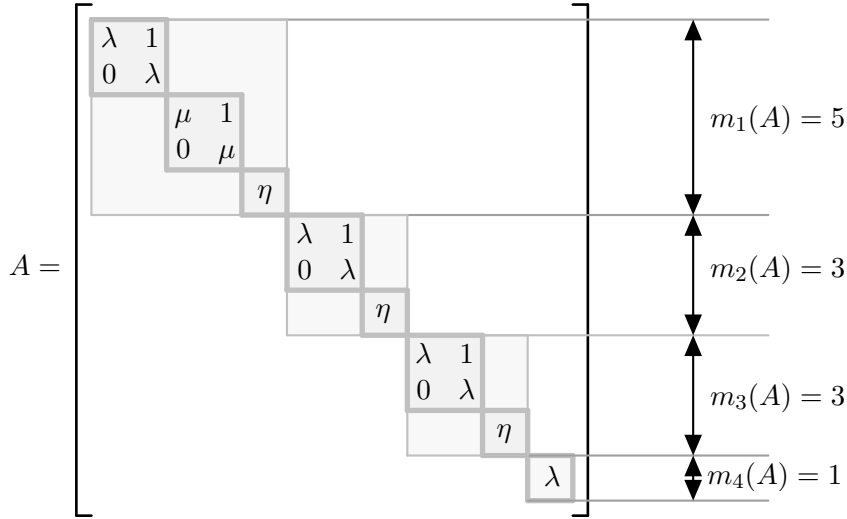
\begin{center}	\vskip-15pt
\definecolor{aqaqaq}{rgb}{0.6274509803921569,0.6274509803921569,0.6274509803921569}
\definecolor{cqcqcq}{rgb}{0.7529411764705882,0.7529411764705882,0.7529411764705882}
\begin{tikzpicture}[line cap=round,line join=round,>=triangle 45,x=1cm,y=1cm]
\clip(-1,-1.65) rectangle (13,5.3);
\fill[line width=2pt,color=cqcqcq,fill=cqcqcq,fill opacity=0.1] (1,4) -- (2,4) -- (2,5) -- (1,5) -- cycle;
\fill[line width=2pt,color=cqcqcq,fill=cqcqcq,fill opacity=0.1] (2,3) -- (3,3) -- (3,4) -- (2,4) -- cycle;
\fill[line width=2pt,color=cqcqcq,fill=cqcqcq,fill opacity=0.1] (3,2.4) -- (3.6,2.4) -- (3.6,3) -- (3,3) -- cycle;
\fill[line width=2pt,color=cqcqcq,fill=cqcqcq,fill opacity=0.1] (3.6,1.4) -- (4.6,1.4) -- (4.6,2.4) -- (3.6,2.4) -- cycle;
\fill[line width=2pt,color=cqcqcq,fill=cqcqcq,fill opacity=0.1] (5.2,-0.2) -- (6.2,-0.2) -- (6.2,0.8) -- (5.2,0.8) -- cycle;
\fill[line width=2pt,color=cqcqcq,fill=cqcqcq,fill opacity=0.1] (4.6,0.8) -- (5.2,0.8) -- (5.2,1.4) -- (4.6,1.4) -- cycle;
\fill[line width=2pt,color=cqcqcq,fill=cqcqcq,fill opacity=0.1] (6.2,-0.8) -- (6.8,-0.8) -- (6.8,-0.2) -- (6.2,-0.2) -- cycle;
\fill[line width=2pt,color=cqcqcq,fill=cqcqcq,fill opacity=0.1] (6.8,-1.4) -- (7.4,-1.4) -- (7.4,-0.8) -- (6.8,-0.8) -- cycle;
\fill[line width=0.8pt,color=cqcqcq,opacity=0.1] (1,2.4) -- (3.6,2.4) -- (3.6,5) -- (1,5) -- cycle;
\fill[line width=2pt,color=cqcqcq,opacity=0.1] (3.6,0.8) -- (5.2,0.8) -- (5.2,2.4) -- (3.6,2.4) -- cycle;
\fill[line width=2pt,color=cqcqcq,opacity=0.1] (5.2,-0.8) -- (6.8,-0.8) -- (6.8,0.8) -- (5.2,0.8) -- cycle;
\draw [line width=2pt,color=cqcqcq] (1,4)-- (2,4);
\draw [line width=2pt,color=cqcqcq] (2,4)-- (2,5);
\draw [line width=2pt,color=cqcqcq] (2,5)-- (1,5);
\draw [line width=2pt,color=cqcqcq] (1,5)-- (1,4);
\draw [line width=2pt,color=cqcqcq] (2,3)-- (3,3);
\draw [line width=2pt,color=cqcqcq] (3,3)-- (3,4);
\draw [line width=2pt,color=cqcqcq] (3,4)-- (2,4);
\draw [line width=2pt,color=cqcqcq] (2,4)-- (2,3);
\draw [line width=2pt,color=cqcqcq] (3,2.4)-- (3.6,2.4);
\draw [line width=2pt,color=cqcqcq] (3.6,2.4)-- (3.6,3);
\draw [line width=2pt,color=cqcqcq] (3.6,3)-- (3,3);
\draw [line width=2pt,color=cqcqcq] (3,3)-- (3,2.4);
\draw [line width=2pt,color=cqcqcq] (3.6,1.4)-- (4.6,1.4);
\draw [line width=2pt,color=cqcqcq] (4.6,1.4)-- (4.6,2.4);
\draw [line width=2pt,color=cqcqcq] (4.6,2.4)-- (3.6,2.4);
\draw [line width=2pt,color=cqcqcq] (3.6,2.4)-- (3.6,1.4);
\draw [line width=2pt,color=cqcqcq] (5.2,-0.2)-- (6.2,-0.2);
\draw [line width=2pt,color=cqcqcq] (6.2,-0.2)-- (6.2,0.8);
\draw [line width=2pt,color=cqcqcq] (6.2,0.8)-- (5.2,0.8);
\draw [line width=2pt,color=cqcqcq] (5.2,0.8)-- (5.2,-0.2);
\draw [line width=2pt,color=cqcqcq] (4.6,0.8)-- (5.2,0.8);
\draw [line width=2pt,color=cqcqcq] (5.2,0.8)-- (5.2,1.4);
\draw [line width=2pt,color=cqcqcq] (5.2,1.4)-- (4.6,1.4);
\draw [line width=2pt,color=cqcqcq] (4.6,1.4)-- (4.6,0.8);
\draw [line width=2pt,color=cqcqcq] (6.2,-0.8)-- (6.8,-0.8);
\draw [line width=2pt,color=cqcqcq] (6.8,-0.8)-- (6.8,-0.2);
\draw [line width=2pt,color=cqcqcq] (6.8,-0.2)-- (6.2,-0.2);
\draw [line width=2pt,color=cqcqcq] (6.2,-0.2)-- (6.2,-0.8);
\draw [line width=2pt,color=cqcqcq] (6.8,-1.4)-- (7.4,-1.4);
\draw [line width=2pt,color=cqcqcq] (7.4,-1.4)-- (7.4,-0.8);
\draw [line width=2pt,color=cqcqcq] (7.4,-0.8)-- (6.8,-0.8);
\draw [line width=2pt,color=cqcqcq] (6.8,-0.8)-- (6.8,-1.4);
\draw (0.9886666666666661,5.1) node[anchor=north west] {$ \begin{matrix} \lambda&1\\ 0&\lambda \end{matrix}$};
\draw (1.9686666666666672,4.1) node[anchor=north west] {$\begin{matrix}\mu&1\\0&\mu\end{matrix}$};
\draw (3.0686666666666684,2.93) node[anchor=north west] {$\eta$};
\draw (3.588666666666669,2.5) node[anchor=north west] {$ \begin{matrix} \lambda&1\\ 0&\lambda \end{matrix}$};
\draw (4.64,1.35) node[anchor=north west] {$ \eta$};
\draw (5.168666666666671,0.9) node[anchor=north west] {$\begin{matrix}\lambda&1\\0&\lambda\end{matrix}$};
\draw (6.25,-0.25) node[anchor=north west] {$ \eta$};
\draw (6.88,-0.8) node[anchor=north west] {$\lambda$};
\draw [line width=1pt] (0.8,5.2)-- (0.8,-1.6);
\draw [line width=1pt] (0.8,-1.6)-- (1,-1.6);
\draw [line width=1pt] (0.8,5.2)-- (1,5.2);
\draw [line width=1pt] (7.6,5.2)-- (7.6,-1.6);
\draw [line width=1pt] (7.6,-1.6)-- (7.4,-1.6);
\draw [line width=1pt] (7.6,5.2)-- (7.4,5.2);
\draw (-0.2313333333333352,2) node[anchor=north west] {$A=$};
\draw [line width=0.8pt,color=aqaqaq] (1,5)-- (10,5);
\draw [line width=0.8pt,color=aqaqaq] (3,2.4)-- (10,2.4);
\draw [line width=0.8pt,color=cqcqcq] (4.6,0.8)-- (10,0.8);
\draw [line width=0.8pt,color=aqaqaq] (6.2,-0.8)-- (10,-0.8);
\draw [line width=0.8pt,color=aqaqaq] (6.8,-1.4)-- (10,-1.4);
\draw [<->,line width=0.8pt] (9,2.4) -- (9,5);
\draw [<->,line width=0.8pt] (9,0.8) -- (9,2.4);
\draw [<->,line width=0.8pt] (9,-0.8) -- (9,0.8);
\draw [<->,line width=0.8pt] (9,-1.4) -- (9,-0.8);
\draw (9.088666666666676,4) node[anchor=north west] {$m_1(A)=5$};
\draw (9.088666666666676,2) node[anchor=north west] {$m_2(A)=3$};
\draw (9.088666666666676,0.3) node[anchor=north west] {$m_3(A)=3$};
\draw (9.048666666666675,-0.75) node[anchor=north west] {$m_4(A)=1$};
\draw [line width=0.8pt,color=cqcqcq] (1,2.4)-- (3.6,2.4);
\draw [line width=0.8pt,color=cqcqcq] (3.6,2.4)-- (3.6,5);
\draw [line width=0.8pt,color=cqcqcq] (3.6,5)-- (1,5);
\draw [line width=0.8pt,color=cqcqcq] (1,5)-- (1,2.4);
\draw [line width=0.8pt,color=cqcqcq] (3.6,0.8)-- (5.2,0.8);
\draw [line width=0.8pt,color=cqcqcq] (5.2,0.8)-- (5.2,2.4);
\draw [line width=0.8pt,color=cqcqcq] (5.2,2.4)-- (3.6,2.4);
\draw [line width=0.8pt,color=cqcqcq] (3.6,2.4)-- (3.6,0.8);
\draw [line width=0.8pt,color=cqcqcq] (5.2,-0.8)-- (6.8,-0.8);
\draw [line width=0.8pt,color=cqcqcq] (6.8,-0.8)-- (6.8,0.8);
\draw [line width=0.8pt,color=cqcqcq] (6.8,0.8)-- (5.2,0.8);
\draw [line width=0.8pt,color=cqcqcq] (5.2,0.8)-- (5.2,-0.8);
\end{tikzpicture}
\end{center}
we look at the {\bf size} of the Jordan blocks for each eigenvalue and group them together in decreasing order to get the bigger blocks. Hence, $m(A)=[
5,3,3,1,0,0,0,0,0,0,0,0]^\top$ whereas the entries of the dual vector are going to be the {\bf number} of bigger blocks with the size greater or equal to $j=1,2,\ldots$, i.e.  $k(A)=[4,3,3,1,1,0,0,0,0,0,0,0]^\top$.

  If $A$ has the same Jordan structure as $A_0$ then we can get the Lipschitz stability of the Schur form. Recall if $A$ has $q$ distinct eigenvalues then
$$\Omega(A)=\{(m_i(A,\lambda_j))_{i=1}^n,j=1,\dots, q\}$$
is called the \textit{Jordan structure} of $A$ and the set of all matrices of the Jordan structure $\Omega$ we denote by $\mathcal{J}(\Omega)$.

\begin{theorem}[Same Jordan structure. Lipschitz forward stability]\label{lipschitz}
	Let $A_0=U_0 T_0U_0^*$. Then, there
	exist constants $K,\varepsilon>0$ $($depending on $A_0$ only$)$ such that for all $A\in\mathcal{J}(\Omega(A_0))$ with $\|A-A_0\|<\varepsilon$ there exists a Schur factorization $U TU^*$ of $A$  such that
	\begin{equation}\label{lipschitzeq}
	\|U-U_0\|+\|T-T_0\|\leq K\|A-A_0\|.
	\end{equation}
	\end{theorem}


If we extend the class of perturbation to those having the same GK numbers as the original matrix, we get another forward stability result, this time with the H\"older type bound.

\begin{theorem}[Same GK numbers. H\"older forward stability]\label{sameGK}
	Let $A_0$ be given. Any Schur canonical form is forward H\"older stable in the class of matrices $A$ having the same Gohberg-Kaashoek numbers as $A_0$. This means that there
	exist  constants $K,varepsilon>0$ $($depending on $A_0$ only$)$ such that for all $A$ with the same GK numbers as $A_0$ and $\|A-A_0\|<\varepsilon$ there exists a Schur factorization $U TU^*$ of $A$ such that the following inequality holds.	\begin{equation}\label{sameGKeq}
	\|U-U_0\|+\|T-T_0\|\leq K\|A-A_0\|^{1/n}.
	\end{equation} 
\end{theorem}

Combining this result with Proposition \ref{GKnonstable}, we completely characterize the forward stability of the Schur form depending on the structure of the perturbation. Moreover, in \cite{MNO} the backward stability result was introduced. Hence, we conclude the investigation of the stability of the Schur canonical form.

{\bf Structure of the paper.}

Section 2 is devoted to some facts from the perturbation theory and theory of gaps/semigaps that are the backbone of our discourse.

You can find proof of Theorem~\ref{lipschitz} in Section 3 and proof of Theorem~\ref{sameGK} in Section 4. We introduce some useful technical lemmas in both Section 3 and Section 4 to aid us in proving main results. 

Note that Theorem~\ref{nonderogatory} (the non-derogatory case) is the direct consequence of Theorem~\ref{sameGK} as discussed in Remark \ref{remark}. However, historically this was the first result we proved before generalizing it. So it deserves its first place among the theorems we present in this paper.

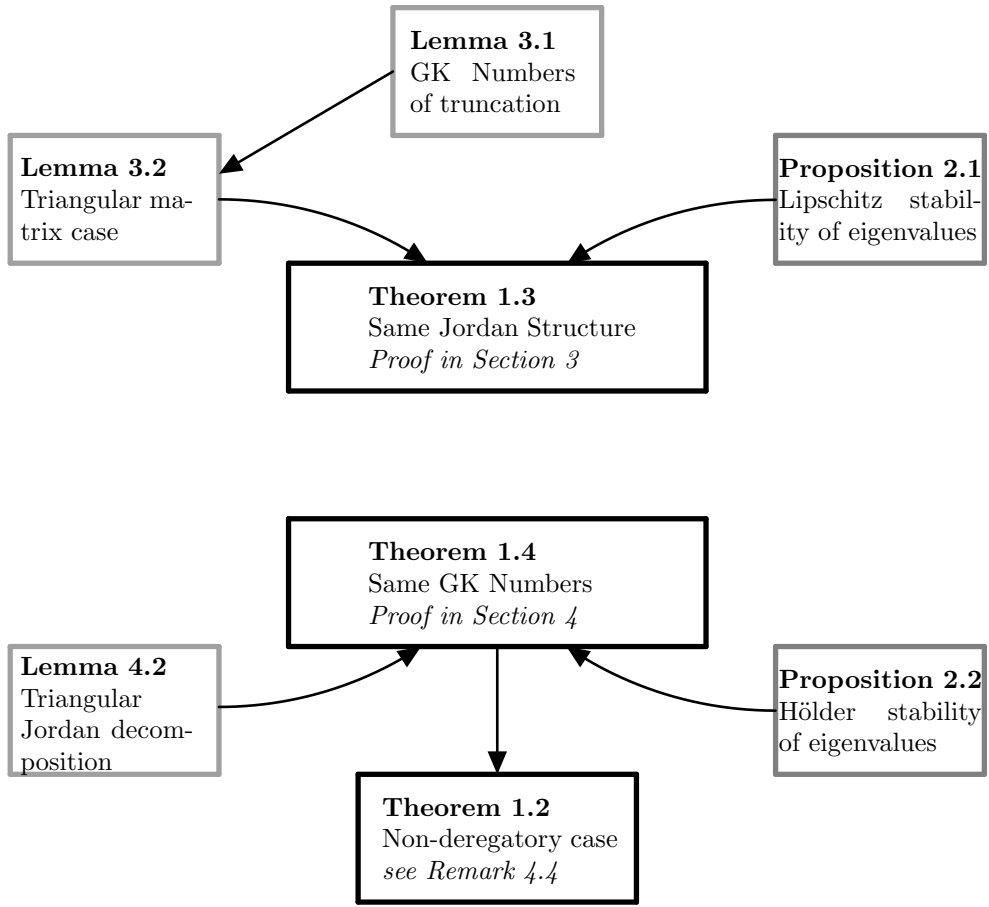
\definecolor{yqyqyq}{rgb}{0.5019607843137255,0.5019607843137255,0.5019607843137255}
\definecolor{aqaqaq}{rgb}{0.6274509803921569,0.6274509803921569,0.6274509803921569}
\noindent{
\begin{tikzpicture}[line cap=round,line join=round,>=triangle 45,x=1.85cm,y=1.7cm]
\clip(-6.6,-2.2) rectangle (1,5.1);
\draw [line width=2pt] (-4.5,3)-- (-4.5,2);
\draw [line width=2pt] (-4.5,2)-- (-1.5,2);
\draw [line width=2pt] (-1.5,2)-- (-1.5,3);
\draw [line width=2pt] (-1.5,3)-- (-4.5,3);
\draw [line width=2pt] (-4.5,1)-- (-4.5,0);
\draw [line width=2pt] (-4.5,0)-- (-1.5,0);
\draw [line width=2pt] (-1.5,0)-- (-1.5,1);
\draw [line width=2pt] (-1.5,1)-- (-4.5,1);
\draw [line width=2pt,color=aqaqaq] (-3.75,5)-- (-3.75,4);
\draw [line width=2pt,color=aqaqaq] (-3.75,4)-- (-2.25,4);
\draw [line width=2pt,color=aqaqaq] (-2.25,4)-- (-2.25,5);
\draw [line width=2pt,color=aqaqaq] (-2.25,5)-- (-3.75,5);
\draw [line width=2pt,color=aqaqaq] (-6.5,4)-- (-6.5,3);
\draw [line width=2pt,color=aqaqaq] (-6.5,3)-- (-5,3);
\draw [line width=2pt,color=aqaqaq] (-5,3)-- (-5,4);
\draw [line width=2pt,color=aqaqaq] (-5,4)-- (-6.5,4);
\draw [line width=2pt,color=yqyqyq] (-1,4)-- (-1,3);
\draw [line width=2pt,color=yqyqyq] (-1,3)-- (0.5,3);
\draw [line width=2pt,color=yqyqyq] (0.5,3)-- (0.5,4);
\draw [line width=2pt,color=yqyqyq] (0.5,4)-- (-1,4);
\draw [line width=2pt,color=aqaqaq] (-6.5,0)-- (-6.5,-1);
\draw [line width=2pt,color=aqaqaq] (-6.5,-1)-- (-5,-1);
\draw [line width=2pt,color=aqaqaq] (-5,-1)-- (-5,0);
\draw [line width=2pt,color=aqaqaq] (-5,0)-- (-6.5,0);
\draw [line width=2pt] (-4,-1)-- (-4,-2);
\draw [line width=2pt] (-4,-2)-- (-2,-2);
\draw [line width=2pt] (-2,-2)-- (-2,-1);
\draw [line width=2pt] (-2,-1)-- (-4,-1);
\draw [line width=2pt,color=yqyqyq] (-1,0)-- (-1,-1);
\draw [line width=2pt,color=yqyqyq] (-1,-1)-- (0.5,-1);
\draw [line width=2pt,color=yqyqyq] (0.5,-1)-- (0.5,0);
\draw [line width=2pt,color=yqyqyq] (0.5,0)-- (-1,0);
\draw (-3.7,4.9) node[anchor=north west] {\parbox{2.2 cm}{  \small  {\bf Lemma \ref{reducingJordan}} \\   GK Numbers of truncation}};
\draw (-1.05,3.9) node[anchor=north west] {\parbox{2.7 cm}{   \small   {\bf Proposition 2.1}\\    Lipschitz stability of eigenvalues}};
\draw (-6.5,3.9) node[anchor=north west] {\parbox{2.3 cm}{   \small   {\bf Lemma \ref{recursion}}\\   Triangular matrix case}};
\draw (-4,2.9) node[anchor=north west] {\parbox{4.3 cm}{ \small   {\bf Theorem 1.3} \\Same Jordan Structure\\  {\it Proof in  Section 3}}};
\draw [<-,shift={(-5,1)},line width=1pt]  plot[domain=0.9272952180016122:1.5707963267948966,variable=\t]({1*2.5*cos(\t r)+0*2.5*sin(\t r)},{0*2.5*cos(\t r)+1*2.5*sin(\t r)});
\draw [->,shift={(-1,1)},line width=1pt]  plot[domain=1.5707963267948966:2.2156816647927595,variable=\t]({1*2.4961023698111062*cos(\t r)+0*2.4961023698111062*sin(\t r)},{0*2.4961023698111062*cos(\t r)+1*2.4961023698111062*sin(\t r)});
\draw [->,line width=1pt] (-3.75,4.5) -- (-5,3.7);
\draw [->,line width=1pt] (-3,0) -- (-3,-1);
\draw [->,shift={(-5,2)},line width=1pt]  plot[domain=4.71238898038469:5.339880761830046,variable=\t]({1*2.4729453202792477*cos(\t r)+0*2.4729453202792477*sin(\t r)},{0*2.4729453202792477*cos(\t r)+1*2.4729453202792477*sin(\t r)});
\draw [<-,shift={(-1,2)},line width=1pt]  plot[domain=4.0679634881877735:4.71238898038469,variable=\t]({1*2.501735490676131*cos(\t r)+0*2.501735490676131*sin(\t r)},{0*2.501735490676131*cos(\t r)+1*2.501735490676131*sin(\t r)});
\draw (-4,0.9) node[anchor=north west] {\parbox{4.3 cm}{ \small   {\bf Theorem 1.4} \\Same GK Numbers\\  {\it Proof in  Section 4}}};
\draw (-3.9,-1.1) node[anchor=north west] {\parbox{4.3 cm}{ \small   {\bf Theorem 1.2} \\Non-deregatory case\\  {\it see Remark 4.4}}};
\draw (-6.5,0) node[anchor=north west] {\parbox{2.3 cm}{   \small   {\bf Lemma 4.2}\\   Triangular Jordan decomposition}};;
\draw (-1.05,-0.1) node[anchor=north west] {\parbox{2.7 cm}{   \small   {\bf Proposition 2.2}\\    H\"older stability of eigenvalues}};
\end{tikzpicture}}

\section{Auxiliary Results}

\subsection{Stability of Eigenvalues}

  Note that  entries on the main diagonal of $T$ in the Schur decomposition are the eigenvalues of $A$. That is why the eigenvalues stability results give us the confidence to consider stability of the Schur decomposition.

The following result can be found in \cite{BOP}.
\begin{proposition} \label{dimeq}
		Let $A_0$ be an $n\times n$ matrix and $\{\lambda_1,\dots,\lambda_n\}$ be its eigenvalues, and $A$ being its perturbation with $\|A-A_0\|<\varepsilon$ for sufficiently small $\varepsilon$ depending on $A_0$ and the eigenvalues $\mu_j$. If the number of eigenvalues of $A_0$ is the same as of $A$, then there is a certain ordering of them such that for some positive $K=K(A_0)$ 
	$$|\mu_i-\lambda_i|\leq K\|A-A_0\|,\quad i=1,2,\dots,|\sigma(A_0)|.$$
\end{proposition}

 For the general case of the eigenvalues stability we have the following result (see~\cite[Appendix K]{Os73}). 
\begin{proposition}\label{eigstab}
	Let $A_0$ be an $n\times n$ matrix and $\{\lambda_1,\dots,\lambda_n\}$ be its eigenvalues. Then, there is an ordering of $\lambda_j$'s that for every $A$ with $\|A-A_0\|<\varepsilon$ for sufficiently small $\varepsilon$ depending on $A_0$ there is an ordering of its eigenvalues $\mu_j$'s and a positive constant $K=K(A_0)$ such that
	\begin{equation}\label{holder}
	|\mu_j-\lambda_j|\leq K\|A-A_0\|^{1/n}.
	\end{equation}\end{proposition}

This type of bounds is called H\"older because of the power $1/n$ for the matrix norm.

 The following example shows that the power $1/n$ in~\eqref{holder} cannot be relaxed and, in general, we can hope only for a H\"older type bound.

 \begin{example}\label{ex1}
 	Consider the following matrices $A_0,A\in \mathbb{C}^{2\times 2}$.
 $$A_0=\begin{bmatrix}
 0&0\\
 1&0
 \end{bmatrix}\text{ and }\ 
 A=\begin{bmatrix}
 0&\epsilon\\
 1&0
 \end{bmatrix}.$$
 \end{example}
 Note that $\|A-A_0\|=\epsilon$. Moreover, $\sigma(A_0)=\{0\}$ and $\sigma(A)=\{\pm\sqrt{\epsilon}\}$. It is easy to see that in this case we have
 $$|0\mp\sqrt{\epsilon}|=\epsilon^{1/2}=\|A-A_0\|^{1/2}.$$
 This example can be easily modified for $n\times n$-matrices.

\subsection{Gap and Semi-gap}

We are going to discuss stability of vectors related to the invariant subspaces and hence need some topological properties of the set of subspaces in $\mathbb{C}^n$. 

Let $\mathcal{M},\mathcal{N}$ be subspaces of $\mathbb{C}^n$.
	A matrix $P_\mathcal{M}$ is called an\textit{ orthogonal projector} onto a subspace $\mathcal{M}\subset \mathbb{C}^n$ if 
	\begin{itemize}
		\item it is surjective:  $\mathrm{Im}P_\mathcal{M}=\mathcal{M}$;
		\item applying it twice yields the same result: $P_\mathcal{M}^2=P_\mathcal{M}$;
		\item it is Hermitian: $P_\mathcal{M}^*=P_\mathcal{M}$.
	\end{itemize}

The following concept is the key definition.

 The \textbf{gap} $\theta(\mathcal{M},\mathcal{N})$ between $\mathcal{M}$ and $\mathcal{N}$ is defined as follows
	$$\theta(\mathcal{M},\mathcal{N})=\|P_\mathcal{M}-P_\mathcal{N}\|$$
	or, equivalently,
	$$\theta(\mathcal{M},\mathcal{N})=\max\left\{\underset{\begin{smallmatrix}
		x\in\mathcal{M}\\ \|x\|=1
		\end{smallmatrix}}{\sup}\underset{y\in\mathcal{N}}{\inf}\|x-y\|,\underset{\begin{smallmatrix}
		y\in\mathcal{N}\\ \|y\|=1
		\end{smallmatrix}}{\sup}\underset{x\in\mathcal{M}}{\inf}\|x-y\|\right\}.$$
Note that $\theta(\mathcal{M},\mathcal{N})$ is a metric on the set of all subspaces in $\mathbb{C}^n$. Moreover, $\theta(\mathcal{M},\mathcal{N})\leq 1$.
The Hausdorff distance between sets $\text{Inv } A$ and  $\text{Inv }  B$ of
all invariant subspaces matrices $A$ and $B$ can be defined as follows
$$\text{dist }(\text{Inv}\,A,\text{Inv}\,B)=\max\{\underset{\mathcal{M}\in\text{Inv }A}{\sup}\theta(\mathcal{M},\text{Inv }B),\underset{\mathcal{N}\in\text{Inv }B}{\sup}\theta(\mathcal{N},\text{Inv }A)\}.
$$ This distance is a metric as well.

We are going to use the following property of gaps between subspaces. It can be found in \cite{GLR86}.
\begin{proposition}\label{orthogap}
	For subspaces  $\mathcal{M},\mathcal{N}\subset\mathbb{C}^n$, we have
	\begin{equation}\label{gap}
	\theta(\mathcal{M},\mathcal{N})=\theta(\mathcal{N}^\perp,\mathcal{M}^\perp).
	\end{equation}
\end{proposition}

The symmetry with respect to subspaces of the gap is actually
a disadvantage.

\begin{proposition}
	Let $\mathcal{M},\mathcal{N}$ be subspaces of $\mathbb{C}^n$.
	\begin{itemize}
		\item [(i)]If $\dim(\mathcal{M})=\dim(\mathcal{N})$ then for any $x\in\mathcal{M}$ there exists a $y\in\mathcal{N}$ such that $\|x-y\|\leq \theta(\mathcal{M},\mathcal{N})$.
		\item[(ii)] If $\dim(\mathcal{M})\neq\dim(\mathcal{N})$ then $ \theta(\mathcal{M},\mathcal{N})=1$.
	\end{itemize}
\end{proposition}

The above result shows us that the gap is often not useful to consider when  $\dim(\mathcal{M})\neq\dim(\mathcal{N})$.
In our theorem, we wish to find bounds on the kernels of the matrices, however, the dimension of the kernels are, in general, not equal. 

The gap provides many useful results in providing a variety of bounds but the
usefulness is limited to when the dimensions are equal. The concept of a semi-gap can be helpful when the dimensions are not equal. This advantage is highly useful when considering matrix
perturbations.

Let $\mathcal{M},\mathcal{N}$ be subspaces of $\mathbb{C}^n$. The quantity
	$$\theta_0(\mathcal{M},\mathcal{N})=\underset{\begin{smallmatrix}
		x\in\mathcal{M}\\ \|x\|=1
		\end{smallmatrix}}{\sup}\underset{y\in\mathcal{N}}{\inf}\|x-y\|$$
	is called the \textit{semigap} (or one-sided gap) from $\mathcal{M}$ to $\mathcal{N}$.

We notice some immediate properties of the semi-gap.

\begin{lemma}\label{semigap}
	Let $\mathcal{M},\mathcal{N}\subset \mathbb{C}^n$ be two subspaces. Then the following statements hold.
	\begin{itemize}
		\item [(i)] $\theta(\mathcal{M},\mathcal{N})=\max\{\theta_0(\mathcal{M},\mathcal{N}),\theta_0(\mathcal{N},\mathcal{M})\}$.
		\item[(ii)] $\theta_0(\mathcal{M},\mathcal{N})=\underset{\begin{smallmatrix}
			x\in\mathcal{M}\\ \|x\|=1
			\end{smallmatrix}}{\sup}{}\|x-P_\mathcal{N}x\|$.
		\item [(iii)]If $\mathcal{N}_1\subset\mathcal{N}_2$, then $\theta_0(\mathcal{M},\mathcal{N}_2)\leq \theta_0(\mathcal{M},\mathcal{N}_1)$,  $\theta_0(\mathcal{N}_1,\mathcal{M})\leq \theta_0(\mathcal{N}_2,\mathcal{M})$.
		\item[(iv)] $\theta_0(\mathcal{M},\mathcal{N})\leq 1$.
		\item[(v)] If $\dim \mathcal{M}>\dim\mathcal{N}$, then $\theta_0(\mathcal{M},\mathcal{N})=1$.
		\item[(vi)] $\theta_0(\mathcal{M},\mathcal{N})<1$ if and only if $\mathcal{M}\cap\mathcal{N}^\perp=\emptyset$.
	\end{itemize} 
\end{lemma}

These facts are well-known and can be found e.g. in \cite{GLR86,K66}.

 To this end we will
need some new results on gap and semigap. These results will be derived next.

\begin{proposition}\label{kernel}
	Let $A_0$ be fixed. Then, there exist $\epsilon, K>0$ such that for all $A$ with $\|A-A_0\|<\epsilon$, we have
	\begin{equation}\label{ker}
	\theta_0(\ker (A),\ker (A_0))\leq K\|A-A_0\|.
	\end{equation}
\end{proposition}

\section{The Same Jordan Structure and Forward Stability
}

Before proving Theorem~\ref*{lipschitz} we need a couple of technical lemmas. 

The next result describes the recursion we will use. In particular, we want to figure out what happens to the GK numbers during each step of recursion.
Here is the idea behind it:
\begin{center}\begin{tikzpicture}[line cap=round,line join=round,x=1cm,y=1cm]
\clip(-12.12,2.5) rectangle (0,5);
\draw (-12.105249153459209,4.8848968485047015) node[anchor=north west] {$m_1(A,\lambda_t)\ge m_2(A,\lambda_t)\ge m_3(A,\lambda_t)\ge \ldots \ge m_{l-1}(A,\lambda_t)\ge m_l(A,\lambda_t)$};
\draw [line width=1pt] (-12,4.2)-- (-12,4);
\draw [line width=1pt] (-12,4)-- (-2.8,4);
\draw [line width=1pt] (-2.8,4)-- (-2.8,4.2);
\draw (-12.0921026287238,3.35) node[anchor=north west] {$m_j(A,\lambda_t)<m_l(A,\lambda_t)\text{ for }j>l$};
\draw [shift={(-11,2)},line width=1pt]  plot[domain=1.5882848543070331:2.1641892860463714,variable=\t]({1*1.5573359986717992*cos(\t r)+0*1.5573359986717992*sin(\t r)+0.1},{0*1.5573359986717992*cos(\t r)+1*1.5573359986717992*sin(\t r)-0.057});
\draw [shift={(-10.369907888385308,5.345025214243993)},line width=1pt]  plot[domain=4.073352330911641:4.695441450578285,variable=\t]({1*1.6752367373558714*cos(\t r)+0*1.6752367373558714*sin(\t r)},{0*1.6752367373558714*cos(\t r)+1*1.6752367373558714*sin(\t r)});
\draw [line width=1pt] (-10.92,3.5)-- (-10.396200937856124,3.5);
\draw [->,line width=1pt] (-10.396200937856124,3.67) -- (-10,3.67);
\draw [->,line width=1pt] (-10.422645098966939,3.5) -- (-10,3.5);
\draw (-9.988658671058467,3.9) node[anchor=north west] {$\text{The corresponding Jordan chains stay the same}.$};
\end{tikzpicture}
\end{center}

So what happens when $m_j(A,\lambda_t)=m_l(A,\lambda_t)$ for some $j$'s greater than $l$? Let $j^*$ be the maximal such index.

\definecolor{cqcqcq}{rgb}{0.7529411764705882,0.7529411764705882,0.7529411764705882}
\begin{tikzpicture}[line cap=round,line join=round,>=triangle 45,x=1cm,y=1cm]
\clip(-14.21201668692858,-1.3) rectangle (1.0090665963473844,1.4648284385882258);
\fill[line width=1pt,color=cqcqcq,fill=cqcqcq,fill opacity=0.05] (-13,-0.043) -- (-13,-0.65) -- (-11.43,-0.65) -- (-11.43,-0.043) -- cycle;
\draw (-13.090581182378063,0) node[anchor=north west] {$m_l(A,\lambda_t)=m_{l+1}(A,\lambda_t)=\ldots=m_{j^*}(A,\lambda_t)>m_l(A,\lambda_t)-1$};
\draw [line width=1pt,color=cqcqcq] (-13,-0.04338353162684747)-- (-13.005786097834683,-0.6589926324169902);
\draw [line width=1pt,color=cqcqcq] (-13.005786097834683,-0.6589926324169902)-- (-11.40848334248257,-0.6392728453138777);
\draw [line width=1pt,color=cqcqcq] (-11.40848334248257,-0.6392728453138777)-- (-11.40848334248257,-0.04570725351019176);
\draw [line width=1pt,color=cqcqcq] (-11.40848334248257,-0.04570725351019176)-- (-13,-0.04338353162684747);
\draw [shift={(-10,-3)},line width=1pt]  plot[domain=1.567619094245166:2.1657704681284176,variable=\t]({1*3.5695681211418098*cos(\t r)+0*3.5695681211418098*sin(\t r)},{0*3.5695681211418098*cos(\t r)+1*3.5695681211418098*sin(\t r)});
\draw [line width=1pt] (-9.98865867105847,0.5695501041069182)-- (-5.238161957918663,0.5754660402378526);
\draw [<-,shift={(-5.780456103254257,-3.2856682745515724)},line width=1pt]  plot[domain=1.007480065302928:1.4312596139915073,variable=\t]({1*3.8996944747832263*cos(\t r)+0*3.8996944747832263*sin(\t r)},{0*3.8996944747832263*cos(\t r)+1*3.8996944747832263*sin(\t r)});
\draw [rotate around={-0.4:(-3,-0.34643400683265624)},line width=1pt,color=cqcqcq,fill=cqcqcq,fill opacity=0.05] (-3.95,-0.37) ellipse (1.35cm and 0.35cm);
\draw (-11.875842296826331,1.2) node[anchor=north west] {$\text{Recursion decreases this chain by one vector.}$};
\end{tikzpicture}
{

Now let us formalize it.

\begin{lemma}\label{reducingJordan}
Consider matrix $B$ with the eigenvalues $\{\lambda_j\}$'s, having the GK numbers $\{m_j(B,\lambda_i)\}$ and $e_1$as its eigenvector corresponding to the Jordan chain for $\lambda_t$ and $m_l(B,\lambda_t)$, i.e. 
	\begin{equation*}
B=\left[\begin{array}{c|c}
\lambda_t&\begin{matrix}
\star&\cdots&\star
\end{matrix}\\ \hline\\
\begin{matrix}
0\\ \vdots \\0
\end{matrix}&\text{\huge $C$}
\end{array}\right].
\end{equation*} 
Then 
\begin{itemize}
    \item $m_j(C,\lambda_i)=m_j(B,\lambda_i)$ for all $i\ne t$ or $i=t$ and $j>l+1$;
    \item $m_l(C,\lambda_t)=m_l(B,\lambda_t)-1$, $m_{l+1}(C,\lambda_t)=m_{l+1}(B,\lambda_t)$ if $m_l(A,\lambda_t)>m_{l+1}(B,\lambda_t)$;
    \item  $m_{j^*}(C,\lambda_t)=m_l(B,\lambda_t)-1$, $m_{j}(C,\lambda_t)=m_{j+1}(B,\lambda_t)$ for $j=l,\dots,j^*-1$ if $m_l(B,\lambda_t)=m_{l+1}(B,\lambda_t)=\ldots=m_{j^*}(B,\lambda_t)$ and $j^*$ is the maximal such index;
    \item $m_j(C,\lambda_t)=m_j(B,\lambda_t)$ if $m_j(B,\lambda_1)<m_{l}(B,\lambda_t)$.
\end{itemize}
\end{lemma}
The proof of this lemma c be found in \cite{MNO}.

The following fact allows us to use recursion in showing the Lipshitz stability of the Schur form under small perturbations preserving the Jordan structure.
Before proceeding let us warn the reader. We need to be cautious and not to fall in the pitfall in front of us. We start with close matrices with the same Jordan structure: $m_1(T_0,0)=m_1(B,0)=2$ and  $m_2(T_0,0)=m_2(B,0)=1$.
\begin{equation*}
T_0=\left[\begin{array}{c|c c}
0&1&0\\ \hline
0&0&0\\
0&0&0
\end{array}\right]
\quad\text{and}\quad
B=\left[\begin{array}{c|c c}
0&1&0\\ \hline
0&0&0\\
0&\varepsilon&0
\end{array}\right].
\end{equation*} 
We might think that after using recursion we might get different Jordan structures for submatrices of $T_0$ and $B$. However, it is not true since $e_1$ is the eigenvector of $T_0$ corresponding to the Jordan chain of length 2. We need to find another eigenvector of $B$ and not just $e_1$ that corresponds to the Jordan chain of length 2 as well. In this case we could take $e_1+\varepsilon e_3$ and build our recursion from there. So the "close" Jordan chains actually are:
$$e_2\to e_1\to0\quad\text{vs.}\quad e_2\to e_1+\varepsilon e_3\to0.$$
Then we construct a unitary matrix where the first column is $\frac{e_1+\varepsilon e_3}{\sqrt{1+\varepsilon^2}}$ that will be close to identity
$$V_1=\begin{bmatrix}
\frac{1}{\sqrt{1+\varepsilon^2}}&0&-\frac{\varepsilon}{\sqrt{1+\varepsilon^2}}\\
0&1&0\\
\frac{\varepsilon}{\sqrt{1+\varepsilon^2}}&0&\frac{1}{\sqrt{1+\varepsilon^2}}
\end{bmatrix}.$$
This gives us the following Schur factorizations.

\begin{equation*}
T_0=I\left[\begin{array}{c|c c}
0&1&0\\ \hline
0&0&0\\
0&0&0
\end{array}\right]I^*
\quad\text{and}\quad
B=V_1\left[\begin{array}{c|c c}
0&\sqrt{1+\varepsilon^2}&0\\ \hline
0&0&0\\
0&0&0
\end{array}\right]V_1^*.
\end{equation*} 
The general result holds as well.
\begin{lemma}\label{recursion}
	Let $T_0$ be an upper triangular matrix with the eigenvalues $\{\lambda_i\}$'s. Then,  for all $B\in\mathcal{J}(\Omega(T_0))$ with the eigenvalues  $\{\mu_j\}$'s  there
exist  a suitable order of $\{\lambda_i\}$'s and $\{\mu_j\}$'s and unitary matrix $V_1$  such that matrices $T_1$ and $B_1$ in
\begin{equation*}
T_0=\left[\begin{array}{c|c}
\lambda_1&\begin{matrix}
\star&\cdots&\star
\end{matrix}\\ \hline\\
\begin{matrix}
0\\ \vdots \\0
\end{matrix}&\text{\huge $T_1$}
\end{array}\right]
\quad\text{and}\quad
V_1^*BV_1=\left[\begin{array}{c|c}
\mu_1&\begin{matrix}
*&\cdots&*
\end{matrix}\\ \hline\\
\begin{matrix}
0\\ \vdots \\0
\end{matrix}&\text{\huge $B_1$}
\end{array}\right]
\end{equation*}
have the same Jordan structure, i.e.  $m_k(T_1,\lambda_l)=m_k(B_1,\mu_l)$. Moreover,
	\begin{equation}\label{IU}
\|I-V_1\|\leq K\|T_0-B\|.
\end{equation}
\end{lemma}
\begin{proof}
	First of all recall that by Proposition~\ref{dimeq} we have the Lipshitz stability of the eigenvalues. That is for some order of $\{\lambda_i\}$'s and $\{\mu_j\}$'s we get
	 	$$|\mu_i-\lambda_i|\leq K\|B-T_0\|,\quad i=1,2,\dots,|\sigma(T_0)|.$$
	So we fix this exact order of the eigenvalues, without loss of generality assume that the first entry of $T_0$ is $\lambda_1$, and construct $\{f_j\}_{j=1}^{m}$, the  Jordan chain corresponding to $\lambda_1$ for $T_0$ and having the eigenvector $f_1=e_1:=[1\ 0\ 0\ \ldots\ 0]^\top$ in it, where $m=m_l(B,\mu_1)=m_l(T_0,\lambda_1)$ for some $l$. 
	
	Our next step is constructing the corresponding "close" chain. 
	According to \cite[Propostition 1.5]{BOP}, there is $\{g_{k}\}$ corresponding to $\mu_1$ of $B$ that obey the desired Lipshitz bound, that is
	$$\|g_j-f_j\|
\le K_j\|B-T_0\|,$$
	where $K_j$ for $j=1,\ldots,m$ depends on $T_0$ only.
	Now, we construct $V_1$ which is close to the identity and unitary whose first column is parallel to $g_1$, i.e. $v_1=\frac{g_1}{\|g_1\|}$. Since $g_1$ is close to $e_1$, the same is true about $v_1$ and $e_1$.
	$$\|v_1-e_1\|=\|v_1-\frac{e_1}{\|g_1\|}+\frac{e_1}{\|g_1\|}-e_1\|\le \frac{\|g_1-e_1\|}{\|g_1\|} +\left|\frac{1}{\|g_1\|}-1\right|\le
	$$
	$$\le \frac{2K_1\|B-T_0\|}{1-K_1\|B-T_0\|}\le\widehat{K}\|B-T_0\|.
$$
	Moreover, 
	the new orthonormal basis $\{{v_1},\ldots,{v_n}\}$ which form $V_1$ satisfies $\|e_i-v_i\|\leq \widehat{K}\|T_0-B\|$. 
	


	Therefore, we have  $\|V_1-I\|\leq \sqrt{n}\widehat{K}\|B-T_0\|$. 
	

	
	
	 
The last question left is whether the Jordan structures of $T_1$ and $B_1$ are still the same. Recall, we applied our argument to Jordan chains of $T_0$ and $B$ of the same length. By Lemma~\ref{reducingJordan} applied to both resulting matrices, we see that $T_1$ and $B_1$ must have the same Jordan structure. 
\end{proof}

Now we have enough machinery to prove Theorem~\ref{lipschitz}.

\begin{proof}[Proof of Theorem~\ref{lipschitz}]
	Let us note that instead of matrices $A_0$ and $A$ we  consider the upper triangular matrix $T_0$ and the matrix $B=U_0^*AU_0$. Here we are using the fact that
	$$\|A-A_0\|=\|A-U_0T_0U_0^*\|=\|U_0(B-T_0)U_0^*\|=\|B-T_0\|,$$
	since $U_0$ is unitary.
	
	Let $\lambda_1,\ldots,\lambda_p$ and $\mu_1,\ldots,\mu_p$ be the distinct eigenvalues of $T_0$ and $B$ respectively. According to Lemma~\ref{recursion} there is a unitary matrix $V_1$ such that $T_1$ and $B_1$ in
	\begin{equation*}
	T_0=\left[\begin{array}{c|c}
	\lambda_1&\begin{matrix}
	\star&\cdots&\star
	\end{matrix}\\ \hline\\
	\begin{matrix}
	0\\ \vdots \\0
	\end{matrix}&\text{\huge $T_1$}
	\end{array}\right]
	\quad\text{and}\quad
	V_1^*BV_1=\left[\begin{array}{c|c}
	\mu_1&\begin{matrix}
	*&\cdots&*
	\end{matrix}\\ \hline\\
	\begin{matrix}
	0\\ \vdots \\0
	\end{matrix}&\text{\huge $B_1$}
	\end{array}\right]
	\end{equation*}
	have the same Jordan structure, i.e.  $m_k(T_1,\lambda_l)=m_k(B_1,\mu_l)$ with
$$	\|I-V_1\|\leq \widehat{K}_1\|T_0-B\|.$$ Moreover,
$$\|T_1-B_1\|\le \|T_0-V_1^*BV_1\|\le \|T_0-T_0V_1+T_0V_1-BV_1+BV_1-V_1^*BV_1\|\le$$
$$\le \|T_0\|\|I-V_1\|+\|V_1\|\|T_0-B\|+\|V_1\|\|I-V_1^*\|\|B\|\le$$
$$\le (\|B\|+\|T_0\|)\|I-V_1\|+\|T_0-B\|\le (\varepsilon+2\|T_0\|)\|I-V_1\|+\|T_0-B\|\le\widetilde{K}_1\|T_0-B\|.$$
	
	Next, we apply Lemma~\ref{recursion} to $T_1$ and $B_1$, i.e. there is a unitary matrix $V_2$ such that
	\begin{equation*}
	T_1=\left[\begin{array}{c|c}
	\lambda_?&\begin{matrix}
	\star&\cdots&\star
	\end{matrix}\\ \hline\\
	\begin{matrix}
	0\\ \vdots \\0
	\end{matrix}&\text{\huge $T_2$}
	\end{array}\right]
	\quad\text{and}\quad
	V_2^*B_1V_2=\left[\begin{array}{c|c}
	\mu_?&\begin{matrix}
	*&\cdots&*
	\end{matrix}\\ \hline\\
	\begin{matrix}
	0\\ \vdots \\0
	\end{matrix}&\text{\huge $B_2$}
	\end{array}\right]
	\end{equation*}
	have the same Jordan structure, i.e.  $m_k(T_2,\lambda_l)=m_k(B_2,\mu_l)$, with
$$
	\|I-V_2\|\leq \widehat{K}_2\|T_0-B\|
$$ and $
\|T_2-B_2\|\le\|T_1-V_2^*B_1V_2\|\leq \widetilde{K}_2\|T_0-B\|.
$ We continue to recursively apply Lemma~\ref{recursion}.

The last step is combining all the matrices. Define  $\widehat{V}_j=I_j\otimes V_j$ and  $\widehat{V}=\widehat{V}_1\cdot\ldots\cdot\widehat{V}_n$, and $T:=\widehat{V}^*B\widehat{V}$. We know that $
\|I-\widehat{V}_j\|\leq \widehat{K}\|T_0-B\|
$ where $\widehat{K}=\max_j \widehat{K}_j$.
Thus, 
$$\|I-\widehat{V}\|=\|I-\widehat{V}_1\cdot\ldots\cdot\widehat{V}_n\|=$$
$$\|I-\widehat{V}_2\cdot\ldots\cdot\widehat{V}_n+\widehat{V}_2\cdot\ldots\cdot\widehat{V}_n-\widehat{V}_1\cdot\ldots\cdot\widehat{V}_n\|\le$$
$$\|I-\widehat{V}_2\cdot\ldots\cdot\widehat{V}_n\|+\|I-\widehat{V}_1\|\le\ldots\le\sum_1^n\|I-\widehat{V}_j\|\le n\widehat{K}\|T_0-B\|$$
and, similarly,
$$\|T_0-T\|\le \widetilde{K}\|T_0-B\|.$$
By taking $K=n\widehat{K}+\widetilde{K}$ we obtain \eqref{lipschitzeq} and hence conclude the proof of this theorem.
\end{proof}


\section{Gohberg-Kaashoek Numbers and Forward Stability}

What happens when the eigenvalues split in general? What do Gohberg-Kaashoek numbers of the original matrix and its perturbation tell us about the forward stability of the Schur form? These are the questions that we are going to address in this section.

First of all, let us notice that Lemma~\ref{reducingJordan} does not work here. Here is why.

\begin{example}
Consider
$$T_0=\left[\begin{array}{cccc|ccc} 
0&1&&&&&\\
&0&1&&&&\\
&&0&1&&&\\ 
&&&0&&&\\\hline
&&&&0&1&\\
&&&&&0&1\\
&&&&&&0\end{array}\right]
\text{ and }B=\left[\begin{array}{cccc|ccc} 
0&1&&&&&\\
&-\varepsilon&1&&&&\\
&&\varepsilon&1&&&\\ 
&&&0&&&\\\hline
&&&&0&1&\\
&&&&&0&1\\
&&&&&&\varepsilon\end{array}\right]
$$
\end{example}
We start with $m_1(T_0)=m_1(T_0,0)=m_1(B,0)+m_1(B,\varepsilon)+
m_1(B,-\varepsilon)=m_1(B)=4$ and $m_2(T_0)=m_2(T_0,0)=m_1(B,0)+m_1(B,\varepsilon)=m_2(B)=3$. If we consider the truncation as before we will get 
$$T_1=\left[\begin{array}{ccc|ccc} 
0&1&&&&\\
&0&1&&&\\ 
&&0&&&\\\hline
&&&0&1&\\
&&&&0&1\\
&&&&&0\end{array}\right]
\text{ and }B_1=\left[\begin{array}{ccc|ccc} 
-\varepsilon&1&&&&\\
&\varepsilon&1&&&\\ 
&&0&&&\\\hline
&&&0&1&\\
&&&&0&1\\
&&&&&\varepsilon\end{array}\right].
$$
That is $m_1(T_0)=3$ and $m_1(B)=4$, $m_2(T_0)=3$ and $m_2(B)=2$. However, it does not mean the result of Theorem~\ref{sameGK} is not true. We just need to find a different approach.

We need an alternative to Jordan bases. Recall that every matrix $A$ is similar to its Jordan form $J$. We could group together all Jordan blocks corresponding to the longest chains for each distinct eigenvalue. Denote by $A_1$ the direct sum of these blocks.Its size is $m_1(A)\times m_1(A)$. Remove these blocks from the consideration and take a look at the Jordan blocks that are left. Now group  all Jordan blocks corresponding to the longest (the second longest for the original set of Jordan blocks) chains for each distinct eigenvalue again. Call $A_2$ their direct sum and its size is $m_2(A)\times m_2(A)$. We continue the process until there is no more Jordan blocks to consider.

Note that $A_1\oplus A_2\oplus\ldots\oplus A_{k_1(A)}$ is permutaion similar to $J$. This construction has several peculiar properties. The characteristic polynomial of $A_1$ is the minimal polynomial of $A$. The characteristic polynomials of $A_k$ are called the invariant factors of $A$ and their degrees are non-increasing. Let us denote them by $p_k(A):=|A_k-\lambda I|$. Two matrices are similar if and only if their invariant factors are identical. 

What we are interested in is the following. Let $P$ be an invertible matrix such that
$A=P(A_1\oplus A_2\oplus\ldots\oplus A_{k_1(A)})P^{-1}$. In this case $P$ consists of a Jordan basis for $A$. Moreover, there is an invertible $S$ such that
$$\widehat{A}_t=\begin{bmatrix}
J_t(\lambda_1)&U&&& \\
&J_t(\lambda_2)&U&&\\
&&\ddots&\ddots&\\
&&&\ddots&U\\
&&&& J_t(\lambda_{s_t})\end{bmatrix},$$
where all the entries of $U$ are zero except the lower left one which is 1, and $A=S(\widehat{A}_1\oplus \widehat{A}_2\oplus\ldots\oplus \widehat{A}_{k_1(A)})S^{-1}$. We will call the basis that forms $S$ {\it the Frobenius basis of $A$.} Note such $S$ exists for any matrix. Since $\lambda_1\ne\lambda_2$, we have
$$\widehat{A}_t=\left[\begin{array}{c|c}
J_t(\lambda_1)&\begin{matrix}&&\\1&\end{matrix} \\ \hline
&J_t(\lambda_2)\end{array}\right]
=R_t\left[\begin{array}{c|c}
J_t(\lambda_1)&\begin{matrix}&&\\0&\end{matrix} \\ \hline
&J_t(\lambda_2)\end{array}\right]R_t^{-1}
=R_t{A}_tR_t^{-1}.$$ 

Let us rewrite it a bit for our purposes.
\begin{lemma}\label{invfactors}
Let $T_0$ be an upper triangular matrix.
$$
T_0=\begin{bmatrix}
\lambda_1&*&\ldots&*\\
0&\lambda_2&\ddots&\vdots\\
\vdots&\ddots&\ddots&*\\
0&\ldots&0&\lambda_n
\end{bmatrix},$$
where $\lambda_j$'s might have repetitions.
Then $T_0=S_0\widehat{J}_0S_0^{-1}$, where $S_0$ is an upper triangular invertible matrix and $J_0$ is also upper triangular and a permutation of the invariant factor form such that $\widehat{J}_0$ has the same main diagonal as $T_0$.
\end{lemma}
\begin{proof}
Let $\{\lambda^{(i)}\}_{i=1}^{d}$ be the set of unique eigenvalues and $\{f_j^{(l)}(\lambda^{(i)})\}$ be a Jordan basis of $T_0$ such that:
\begin{itemize}
    \item $f_j^{(l)}(\lambda^{(i)})=(T_0-\lambda^{(i)})f_{j+1}^{(l)}(\lambda^{(i)})$ for $l=1,\ldots,k_1(T_0,\lambda^{(i)})$ and $j=1,\ldots,m_l(T_0,\lambda^{(i)})-1$, 
    \item $f_{m_l(T_0,\lambda^{(i)})}^{(l)}(\lambda^{(i)})\in \ker(T_0-\lambda^{(i)})^{m_l(T_0,\lambda^{(i)})}\setminus(\ker(T_0-\lambda^{(i)})^{m_l(T_0,\lambda^{(i)})-1}\cup{\text{span}}\{f_j^{(s)}(\lambda^{(i)})\}_{\tiny\begin{matrix}j=1,\ldots,m_s(T_0,\lambda^{(i)})\\s=1,\ldots,l-1.
\end{matrix}})$. 
\end{itemize}
We will require some additional restrictions on $\{\lambda^{(i)}\}_{i=1}^{d}$ that we are about to discuss.

Recall that $T_0$ is upper triangular which means that if \begin{equation}\label{eigenvector}
    f_j^{(l)}(\lambda^{(i)})= \sum_{t=1}^s \alpha_te_t
\end{equation} with $\alpha_s\ne0$ then $\lambda_s=\lambda^{(i)}$ since that will correspond to non-pivot columns of $T_0-\lambda^{(i)}I$. Moreover, we want $S_0$ to be upper triangular. That is for each $\lambda^{(i)}$ while solving $(T_0-\lambda^{(i)}I)x=0$ to get the eigenvectors $\{f_1^{(l)}(\lambda^{(i)})\}$ we set one free variable to 1 and the rest to 0 at a time. Similarly, while solving $(T_0-\lambda^{(i)}I)x=f_1^{(l)}(\lambda^{(i)})$ to get the generalized eigenvectors $\{f_1^{(l)}(\lambda^{(i)})\}$ we set all free variable to 0 and so on until we exhaust all the chains corresponding to $\lambda^{(i)}$. 

This will guarantee the minimal possible $s$ for each $f_j^{(l)}(\lambda^{(i)})$ in \eqref{eigenvector}. That is for each $s=1,\ldots,n$ there is exactly one such a vector $f_j^{(l)}(\lambda^{(i)})$ for some $i,j$, and $l$ among our Jordan basis  that  decomposition \eqref{eigenvector} with $\alpha_s\ne0$ holds true.

Next consider $v_j=\sum_{t=1}^{d}f_{m_j(T_0,\lambda_t)}^{(j)}(\lambda_t)$ for $j=1,\ldots,k_1(T_0)$. For each $j$ construct $$v_j\to(T_0-\lambda_{s_1})v_j\to(T_0-\lambda_{s_2})(T_0-\lambda_{s_1})v_j\to\ldots\to\left(\prod_{t=1}^{m_j(T_0)}(T_0-\lambda_{s_{m_j(T_0)-t}})\right)v_j$$ where each $s_j$ corresponds to the largest index in the decomposition in the standard basis of the previous vector in the chain. That is $s_1>s_2>\ldots>s_{m_j(T_0)}$. Note $\left(\prod_{t=1}^{m_j(T_0)}(T_0-\lambda_{s_{m_j(T_0)-t}})\right)v_j=p_j(T_0)v_j=0$. By our construction all the indexes $s_t$'s are coming from the Jordan chains involving $\{f_{m_j(T_0,\lambda_t)}^{(j)}(\lambda_t)\}_{t}$. 

We are ready to construct $S_0$ and $\widehat{J}_0$. For each $j$ we put vector $v_j$ as $s_1$th column of $S$ and all the vectors in the chain we described above at their corresponding $s_t$th columns of $S$. In $\widehat{J}_0$ we have the only non-zero entry as $\lambda_{s_{m_j(T_0)}}$ in  the $s_{m_j(T_0)}$th column which is the main diagonal entry. Each $s_t$th column for $t=1,\ldots,{m_j(T_0)}-1$ consists of $\lambda_{s_t}$ on the main diagonal entry and 1 as $(s_t,s_{t+1})$ entry, the rest entries are 0.

Hence, we got $S_0$, an upper triangular matrix with no zeros on the main diagonal (i.e. invertible), and $\widehat{J}_0$, an upper triangular matrix which is a row-column permutation of the invariant form of $T_0$. Moreover, by our construction $T_0S_0=S_0\widehat{J}_0$ or $T_0=S_0\widehat{J}_0S_0^{-1}$ which is exactly what we wanted to show.
\end{proof}
\begin{proof}[Proof of Theorem~\ref{sameGK}]
Let us start with the given Schur decomposition of $A_0$, that is $A_0=U_0T_0U_0^*$ where $U_0$ is unitary and $T_0$ is upper triangular. As before we consider matrices $T_0$ and $B=U_0^*AU_0$. From the discussion in this chapter we know that $T_0=S_0\widehat{J}_0S_0^{-1}$ and $B=S\widehat{J}S^{-1}$, where $\widehat{J}_0$ and $\widehat{J}$  are the direct sums of invariant factors and $S_0$ and $S$ consist of the invariant factors bases of $T_0$ and $B$ respectively.

According to Lemma \ref{invfactors}, $\widehat{J}_0$ and $S_0$ could be chosen to be upper triangular, since $T_0$ is upper triangular.

Moreover, we can make these factors close to each other. Since we have control over $\varepsilon$, we can guarantee the separation of eigenvalues, i.e. it is not only that the GK numbers for matrices are the same, but also each eigenvalue $\lambda_j$ of $T_0$ has exactly the same number of the eigenvalues $\mu_{j_1},\ldots\mu_{j_?}$ of $B$ counting multiplicities as the multiplicity of $\lambda_j$ and so $m_k(T_0,\lambda_j)=m_k(B,\mu_{j_1})+\ldots+m_k(B,\mu_{j_?})$.
 That is why there is an order of $\lambda$'s and $\mu$'s such that 
 $$\|\widehat{J}_0-\widehat{J}\|\le \sum_{i=1}^n|\lambda_?-\mu_?|\le K_0\|A_0-A\|^{1/n}.$$
 
	Let us decompose $S_0$ into the basis of $T_0$ described in Lemma \ref{invfactors} and pick $\{f_j\}_{j=1}^{m}$, a chain from this basis, where $m=m_l(B)=m_l(T_0)$. 	
	Our next step is constructing the corresponding "close" chain for $\{\mu_j\}$'s. For
	$$f_m\overset{T_0-\lambda_m I}{\longrightarrow}f_{m-1}\overset{T_0-\lambda_{m-1} I}{\longrightarrow}\ldots\overset{T_0-\lambda_{2} I}{\longrightarrow}f_{1}\overset{T_0-\lambda_1I}{\longrightarrow}0$$
	we find a close to $f_m$ vector $g_m$, using Proposition \ref{kernel},
		$$\|  f_m-g_m\|\le \theta\left( \prod_{j=1}^m\left(T_0 -\lambda_jI\right),\prod_{j=1}^m \left(B -\mu_jI\right)\right)\le K_m\| T_0-B\|^{1/n}$$ 
		and construct a
	chain $\{g_j\}_{j=1}^{m}$ consisting of vectors $g_k=(B-\mu_{k+1})g_{k+1}$ with $\|\mu_{k+1}-\lambda_{k+1}\|<K\|B-T_0\|^{1/n}$. They satisfy the following relations
	
	$$\|g_k-f_k\|=\|(B-\mu_{k+1}I)g_{k+1}-(T_0-\lambda_{k+1} I)f_{k+1}\|=$$
	$$\|(B-T_0)g_{k+1} +T_0(g_{k+1} -f_{k+1} )-(\mu_{k+1}-\lambda_{k+1})g_{k+1} -\lambda_{k+1}(g_{k+1} -f_{k+1} )\|\le$$
	$$\|B-T_0\|\|g_{k+1} \|+\|T_0\|\|g_{k+1} -f_{k+1} \|+|\mu_{k+1}-\lambda_{k+1}|\|g_{k+1} \|$$
	$$+|\lambda|\|g_{k+1} -f_{k+1} \|\le K_{k}\|B-T_0\|^{1/n},$$
	where $K_j$'s ($j=1,\ldots,m$) still depend on $T_0$ only and the power $1/n$ appears due to Proposition~\ref{eigstab}. Define $\widetilde{K}=\max\{K_j\}$. That is 
	$$\|g_k-f_k\|\le \widetilde{K}\|B-T_0\|^{1/n}$$
	for all $k=1,\ldots,m$.
	We use $g$'s to construct    $S$ close to $S_0$, i.e. $$\|S-S_0\|\le n\widetilde{K}\|B-T_0\|^{1/n}.$$
	
	\begin{remark}\label{QR} The QR decompositions are Lipschitz stable in the class of invertible matrices, i.e. if $S_0=Q_0R_0$ is invertible then for any $S=QR$ sufficiently close to $A$ we have
	$$
	\|Q-Q_0\|+\|R-R_0\|\le K\|S-S_0\|.$$
	This fact together with its proof could be found in  \cite{B} for example.
	\end{remark}	

	Let us fix the following QR decomposition of $S_0$: $S_0=IS_0$, where $I$ is the identity matrix. Then according to Remark~\ref{QR}, since $S_0$ is invertible, we can find a QR decomposition of $S=QR$ such that
	\begin{equation}
	\|Q-I\|+\|R-S_0\|\le K\|S-S_0\|\le nK\widetilde{K}\|B-T_0\|^{1/n}.
	\end{equation}

If we define matrix $U=U_0Q$ then we get 
\[ 
 \|U_0-U\|=\|I-Q\|\le nK\widetilde{K}\|A_0-A\|^{1/n}
\] for some $K_1>0$. Moreover, $T:=U^*AU=R\widehat{J}R^{-1}$ is upper triangular as a product of upper triangular matrices with
$$
 \|T_0-T\|=\|U_0^*A_0U_0-U^*AU\|\le \|U_0^*A_0U_0-U_0^*A_0U\|+
 $$
 $$+\|U_0^*A_0U-U_0^*AU\|+\|U_0^*AU-U^*AU\|\le \|A_0\|\|U_0-U\|+
 $$
 $$+\|A_0-A\|+\|U_0^*-U^*\|\|A\|\le \widehat{K}K\|A-A_0\|^{1/n}.$$
Hence, \eqref{sameGKeq} holds true.

\end{proof}

\begin{remark}\label{remark}
	Note that instead of asking the question: {\rm \lq\lq For which class of matrices can we guarantee the forward stability of the Schur form?\rq\rq} we could ask ourselves: {\rm \lq\lq Which class of matrices has the same GK numbers as any of its perturbations in a small enough neighborhood?\rq\rq} If an $n\times n$ matrix $A_0$ is non-derogatory then  $m(A_0)=[n,0,0,\ldots,0]^\top$ and $k(A_0)=[1,1,1,\ldots,1]$. We will have the same picture for any of its permutations $A$ with $\|A-A_0\|<\epsilon$ for small enough $\epsilon$ depending on $A_0$ only (see \cite{O89}).
	
	Hence, Theorem~\ref{nonderogatory} is the direct consequence of Theorem~\ref{sameGK}. 
\end{remark}

To summarize, we have characterized the class of perturbations of matrices for which Schur forms are forward stable. Thus, completing the study of the Schur form stability under a small perturbation.

\end{document}